\documentclass[11pt,a4paper]{amsart}
\usepackage{amsmath,amsfonts,amssymb,amsthm}
\usepackage{mathrsfs}
\usepackage[all,cmtip]{xy}

\usepackage[left=.541in,right=.541in,top=0.6in,bottom=0.6in]{geometry}

\newtheorem{theorem}{Theorem}[section]
\newtheorem{lemma}[theorem]{Lemma}
\newtheorem{proposition}[theorem]{Proposition}
\newtheorem{fact}[theorem]{Fact}
\newtheorem*{fact*}{Fact}

\theoremstyle{definition}
\newtheorem{example}[theorem]{Example}
\newtheorem{definition}[theorem]{Definition}
\newtheorem{remark}[theorem]{Remark}
\newtheorem{question}[theorem]{Question}

\newcommand{\Z}{{\mathbb Z}}
\newcommand{\N}{{\mathbb N}}

\renewcommand{\S}{{\mathbb{S}}}

\def\kal#1{\mathcal{K}_{#1}}   
\def\hull#1{\langle{#1}\rangle}

\title{Topologically independent sets in precompact groups}

\author[J. Sp\v{e}v\'ak]{Jan Sp\v{e}v\'ak}
\address[Jan Sp\v{e}v\'ak]{Department of Mathematics\\ Faculty of Science\\ J. E. Purkyne
University, \v{C}esk\'{e} ml\'{a}de\v{z}e 8,
400 96 \'{U}st\'{i} nad Labem\\
Czech Republic} \email{jan.spevak@ujep.cz}

\keywords{Topological group, precompact group, infinite direct sum, topologically independent set, absolutely Cauchy summable set}

\subjclass[2010]{Primary 22C05; Secondary 20K25}

\begin{document}

\begin{abstract}
It is a simple fact that a subgroup generated by a subset $A$ of an abelian group is the direct sum of the cyclic groups $\hull{a}$, $a\in A$ if and only if the set $A$ is independent. In \cite{DSS} the concept of an {\em independent} set in an abelian group was generalized to a {\em topologically independent set} in a topological abelian group (these two notions coincide in discrete abelian groups). It was proved that
a topological subgroup generated by a subset $A$ of an abelian topological group is the Tychonoff direct sum of the cyclic topological groups $\hull{a}$, $a\in A$ if and only if the set $A$ is topologically independent and absolutely Cauchy summable.
Further, it was shown, that the assumption of absolute Cauchy summability of $A$ can not be removed in general in this result. In our paper we show that it can be removed in precompact groups. 

In other words, we prove that if $A$ is a subset of a {\em precompact} abelian group, then the topological subgroup generated by $A$ is the Tychonoff direct sum of the topological cyclic subgroups $\hull{a}$, $a\in A$ if and only if $A$ is topologically independent. We show that precompactness can not be replaced by local compactness in this result.
\end{abstract}

\maketitle

{\em All groups in this paper are assumed to be abelian and all topological groups are assumed to be Hausdorff.} A~topological group is {\em precompact} if it is a topological subgroup of a compact group. As usually, the symbols $\N$ and $\Z$ stay for the sets of natural numbers and integers respectively.

Given an abelian group $G$, by $0_G$ we denote the zero element of $G$, and the subscript is omitted when there is no danger of confusion. Given a subset $A$ of $G$, the symbol $\hull{A}$ stays for the subgroup of $G$ generated by $A$. For ${a}\in G$, we use the symbol $\hull{a}$  to denote $\hull{\{a\}}$. Following \cite{DSS}, the symbol $S_A$ stays for the direct sum 
$$S_A=\bigoplus _{a\in A} \hull{a},$$
and by $\kal{A}$ we denote the unique group homomorphism 
$$\kal{A}: S_A\to G$$
which extends each natural inclusion map $\hull{a}\to G$ for $a\in A$. As in \cite{DSS}, we call the map $\kal{A}$ the {\em Kalton map associated with $A$\/}.

 We say that $\hull{A}$ {\em is the direct sum of cyclic groups $\hull{a}$, $a\in A$} provided that the Kalton map $\kal{A}$ is an isomorphic embedding. When $G$ is a topological group, we always consider $\hull{a}$  with the subgroup topology inherited from $G$ and $S_A$ with the subgroup topology inherited from the Tychonoff product $\prod_{a\in A}\hull{a}$. Finally, we say that  $\hull{A}$ {\em is a Tychonoff direct sum of cyclic groups $\hull{a}$, $a\in A$} if the Kalton map $\kal{A}$ is at the same time an isomorphic embedding and a homeomorphic embedding.

\section{Introduction}
The concept of compactness allows to transfer some purely non-topological issues into the realm of topology. A nice example of this phenomenon is the paper of Nagao and Shakhmatov  (see \cite{NS}), where the classical, purely combinatorial result of Landau on the existence of kings in finite tournaments, where finite tournament means a finite directed complete graph, is generalized by means of continuous weak selections to continuous tournaments for which the set of players is a compact Hausdorff space. In our paper we provide another example of this phenomenon which non-trivially transfers a result from the area of abelian groups to the realm of precompact abelian groups.

Recall that
a subset $A$ of nonzero elements of a group $G$ is {\em independent} provided that for every finite set $B\subset A$ and every family $(z_a)_{a\in B}$ of integers the equality $\sum_{a\in B}z_aa=0$ implies $z_aa=0$ for all $a\in B$. 

Similarly, a subset $A$ of nonzero elements of a topological group $G$ is {\em topologically independent\/} (see \cite[Definition 4.1]{DSS}) provided that for every neighborhood $W$ of $0_G$ there exists neighborhood $U$ of $0_G$ such that for every finite set $B\subset A$ and every  family $(z_a)_{a\in B}$ of integers the inclusion $\sum_{a\in B}z_aa\in U$ implies $z_aa\in W$ for all $a\in B$. This neighborhood $U$ is called a {\em $W$-witness\/} of the topological independence of $A$.

One can readily verify that in (Hausdorff) topological groups every topologically independent set is independent (see \cite[Lemma 4.2]{DSS}) and that these two notions coincide in discrete groups. Thus topological independence can be viewed as a natural generalization of independence.

 Let us recall a basic and simple fact about independent sets.
\begin{fact*}
A set $A$ of nonzero elements of a group is independent if and only if the subgroup generated by $A$ is a direct sum of the cyclic groups $\hull{a}$, $a\in A$.
\end{fact*}
The aim of this paper is to prove the following counterpart of the above fact. Its proof is postponed to the end of the next section.

\begin{theorem}\label{thm:main}
A set $A$ of nonzero elements of a {precompact} group is {topologically} independent if and only if the { topological} subgroup generated by $A$ is a Tychonoff direct sum of the cyclic topological groups $\hull{a}$, $a\in A$.
\end{theorem}

Example \ref{ex:loc:comp:does:not:work} demonstrates, that precompactness can not be replaced by local compactness in Theorem~\ref{thm:main}.

In \cite{DSS} a result closely related to Theorem \ref{thm:main} was obtained. In order to state it, recall that by \cite[Definition 3.1]{DSS} a subset $A$ of a topological group is {\em absolutely Cauchy summable} provided that for every neighborhood $U$ of $0_G$ 
\begin{equation}\label{eq:def:of:cauch:summable}
\mbox{there exists finite $F\subset A$ such that }\hull{A\setminus F}\subset U.
\end{equation} 
Let us state the promised result (see \cite[Theorem 5.1]{DSS}).

\begin{fact}\label{fact:basic}
A subset $A$ of nonzero elements of a topological group $G$ is at the same time topologically independent and absolutely Cauchy summable if and only if the { topological} subgroup generated by $A$ is a Tychonoff direct sum of the cyclic topological groups $\hull{a}$, $a\in A$.
\end{fact}
In view of Fact \ref{fact:basic} we can see  that Theorem \ref{thm:main} reads as follows: {\em Every topologically independent set in a precompact group is absolutely Cauchy summable}. Let us note, that absolutely Cauchy summable sets can be far away from topologically independent sets even in the realm of compact groups. Indeed, by \cite[Remark 5.3]{DSS}, every null sequence in the compact metric group $\Z_p$ of $p$-adic integers is absolutely Cauchy summable while every topologically independent subset of $\Z_p$ is a singleton.

The ``only if'' part of Fact \ref{fact:basic} is based on two straightforward results. First one states that absolute Cauchy summability of a set $A$ is equivalent to the continuity of the Kalton map $\kal{A}$ (\cite[Theorem 3.5]{DSS}), while the second says that topological independence of a set $A$ implies that the Kalton map $\kal{A}:S_A\to\hull{A}$ is an open isomorphism (\cite[ Lemma 4.2 and Proposition 4.7 (i)]{DSS}). One may ask, whether the implication in the latter statement can be reversed:

\begin{question}\label{Q1}
Let $A$ be a subset of nonzero elements of a topological group $G$ such that the Kalton map $\kal{A}:S_A\to \hull{A}$ is an open isomorphism. Must $A$ be topologically independent?
\end{question}

 Example \ref{ex:open:iso:not:enough} provides a negative answer to this question even in the case, when $G$ is precompact.

\section{Precompact group topologies on direct sums}
By $\S$ we denote the unit circle in the complex plain  which is a compact group with respect to multiplication of complex numbers. 

Given a group $G$ we denote by $G'$ its character group (the group of all homomorphisms from $G$ to $\S$). Elements of $G'$ are called {\em characters}. If $G$ is a topological group, then the symbol $G^*$ stays for the subgroup of $G'$ consisting of all continuous characters.

Let $H$ be a group of characters on $G$. We denote by $T_H$ the coarsest group topology on $G$ making all characters of $H$ continuous. Since the collection $\{V_\delta:\delta\in(0;\pi)\}$, where  $$V_\delta=\{e^{it}:t\in(-\delta,\delta)\},$$ is a local base of the topology of $\S$ at the identity, the topology $T_H$ has a local base at $0_G$ consisting of all the sets of the form 
\def\U#1#2{\mathcal{U}(\chi_1,\ldots,\chi_#1;#2)}
$$\U{n}{\delta}=\{g\in G:\chi_i(g)\in V_\delta \mbox{ for all } i=1,\ldots,n\},$$
where $\delta\in(0;\pi)$, $n\in\N$ and $\chi_1,\ldots,\chi_n\in H$. The topology is Hausdorff whenever the characters of $H$ separate points of $G$.

Let us recall a basic fact about the topology $T_H$ for precompact groups (see \cite[Theorem 2.3.2]{DPS}).
\begin{fact}\label{fact:precompact:top}
Let $G$ be a topological group. Then the topology of $G$  is precompact if and only if it is equal to $T_{G^*}$ and the characters of $G^*$ separate points of $G$.
\end{fact}

\begin{definition}
Given a subset $A$ of a group $G$ and a character $\chi\in G'$ we define the {\em $A$-support} of $\chi$ by the formula $$supp_A(\chi)=\{a\in A: \chi(a)\neq 1\}.$$
We say that the character $\chi$ is {\em finitely $A$-supported} provided that $supp_A(\chi)$ is finite.
\end{definition}

\begin{lemma}\label{lemma:basic}
Let $A$ be an infinite subset of a compact group $G$, and  $V$  an open neighborhood of the identity element of $G$. Then there are distinct $a,b\in A$ such that $ab^{-1}\in V$.
\end{lemma}
\begin{proof}
Let $U$ be an open neighborhood of the identity element of $G$ such that $U^{-1}U\subset V$.  Since $G$ is compact, there are $a,b\in A$ such that
\begin{equation}
\label{eq:lemma}
Ua\cap Ub\neq\emptyset.
\end{equation}
Otherwise $A$ would be an infinite closed discrete subset of a compact group $G$ - a contradiction.
Now, \eqref{eq:lemma} yields    $ab^{-1}\in U^{-1}U\subset V$.
\end{proof}

The following lemma is obvious.
\begin{lemma}\label{lemma:not:top:indep}
Let $A$ be a subset of  a topological group  $G$  and  $\mathcal{V}$  a fixed local base at $0_G$ of the topology of $G$. Assume that $W$ is a neighborhood of $0_G$ such that no element of $\mathcal{V}$ is a $W$-witness of the topological independence of $A$. Then $A$ is not topologically independent.
\end{lemma}
Our next proposition plays the key role in the proof of Theorem \ref{thm:main}.
\begin{proposition}\label{prop:finitely:supported:characters}
Let $A$ be a topologically independent subset of a precompact group $G$. Then each character of $G^*$ is finitely $A$-supported. 
\end{proposition}
\begin{proof}
We will prove the contrapositive. In order to do so, let $supp_A(\chi)$ be infinite for some $\chi\in G^*$. Put $W=\mathcal{U}(\chi;\frac{\pi}{2})$ and pick $n\in\N$, $\chi_1,\ldots,\chi_n\in H$ and $\delta\in(0;\pi)$ arbitrarily. By Lemma \ref{lemma:not:top:indep}, it suffices to show that $\U{n}{\delta}$ is not a $W$-witness of the topological independence of $A$.

 For every $a\in supp_A(\chi)$ we have $\chi(a)\neq 1$. Therefore, we can find $z_a\in\Z$ such that $\chi(z_aa)=\chi(a)^{z_a}\not\in V_\frac{\pi}{2}$. Consequently,  
\begin{equation}\label{zaanotinW}
z_aa\not\in W \mbox{ for all } a\in supp_A(\chi).
\end{equation}
To finish the proof, it remains to find distinct $a,b\in supp_A(\chi)\subset A$ such that 
\begin{equation}\label{eq:neco}
z_aa-z_bb\in\U{n}{\delta}.
\end{equation}
Indeed, in this case $\U{n}{\delta}$ is not a $W$-witness of the topological independence of $A$ by \eqref{zaanotinW}.
If there are distinct $a,b\in supp_A(\chi)$ such that $\chi_i(z_aa)=\chi_i(z_bb)$ for all $i=1,\ldots,n$, then \eqref{eq:neco} holds. Otherwise the set $\{s_a:a\in supp_A(\chi)\}$, where $$s_a=(\chi_1(z_aa),\ldots,\chi_n(z_aa)),$$ is an infinite subset of the compact group $\S^n$.  The set $V=\underbrace{V_{{\delta}}\times\ldots\times V_{{\delta}}}_n$ is an open neighborhood of the identity element $(\underbrace{1,\ldots,1}_n)$ of $\S^n$. Thus, by Lemma \ref{lemma:basic}, there are distinct $a,b\in supp_A(\chi)$ such that $s_as_b^{-1}\in V$. Therefore, $$\chi_i(z_aa-z_bb)\in V_\delta \mbox{ for all } i=1,\ldots,n.$$ This yields \eqref{eq:neco} and finishes the proof.
\end{proof} 

We omit the straightforward proof of the next simple lemma.
\begin{lemma}\label{lemma:subbase:suffices}
Let $A$ be a subset of a topological group $G$ and $\mathcal{V}$ a local subbase at $0_G$ of the topology of $G$. Then $A$ is absolutely Cauchy summable if and only if   \eqref{eq:def:of:cauch:summable} holds for every $U\in\mathcal{V}$.
\end{lemma}

\begin{proposition}\label{proposition:top:ind:is:acs}
Let $A$ be a topologically independent subset of a precompact group $G$. Then $A$ is absolutely Cauchy summable.
\end{proposition}
\begin{proof}
By Fact \ref{fact:precompact:top}, the set  $\{\mathcal{U}(\chi;\delta):\chi\in G^*,\delta\in(0;\pi)\}$ is a local subbase at $0_G$ of the topology of $G$. Fix arbitrary $\delta\in(0;\pi)$, $\chi\in G^*$, and put $U=\mathcal{U}(\chi;\delta)$. By Lemma \ref{lemma:subbase:suffices}, it suffices to show \eqref{eq:def:of:cauch:summable}. Put $F=supp_A(\chi)$. Then $F$ is finite by Proposition \ref{prop:finitely:supported:characters}. Observe that \eqref{eq:def:of:cauch:summable} is satisfied.
\end{proof}

Now we are in position to prove Theorem \ref{thm:main}: 

\begin{proof}[\bf Proof of Theorem \ref{thm:main}]
The proof follows immediately from Proposition \ref{proposition:top:ind:is:acs} and Fact \ref{fact:basic}.
\end{proof}

\section{Examples and remarks}
\begin{example}\label{ex:loc:comp:does:not:work}
{\em Let  $A$ be an arbitrary infinite independent subset of a discrete (in particular, locally compact) group $G$. Then $A$ is topologically independent but the topological subgroup  $\hull{A}$ is never a Tychonoff direct sum of  cyclic topological subgroups $\hull{a}$, $a\in A$.} Indeed, as was noted in the Introduction, independent sets coincide with topologically independent sets in discrete groups. On the other hand a  Tychonoff direct sum of infinitely many non-trivial topological groups is always non-discrete while every topological subgroup of a discrete group is discrete. 
\end{example}

\begin{remark}
Recall that a topological group is {\em compactly generated} if it contains a compact subset $A$ such that $\hull{A}$ is the whole group. Let us note that the discrete group $G$ from Example \ref{ex:loc:comp:does:not:work} is not compactly generated  as it contains infinite independent subset and consequently each its generating set is infinite and discrete (thus non-compact). We do not know whether there is an infinite topologically independent subset $A$ of a  locally compact compactly generated group $G$ such that $\hull{A}$ is not a Tychonoff direct sum of the cyclic groups $\hull{a}$ $a\in A$. In other words, we do not know whether the word ``precompact'' can be replaced by ``locally compact compactly generated'' in Theorem \ref{thm:main}.
\end{remark}

Our next example provides a negative answer to Question \ref{Q1}.

\begin{example}\label{ex:open:iso:not:enough}
{\em There exists a subset $A$ of a precompact group such that the Kalton map $\kal{A}:S_A\to\hull{A}$ is an open isomorphism, but $A$ is not topologically independent.} To show this, let  $A$ be an arbitrary infinite topologically independent subset of a precompact group $G$ and let $\tau$ denote the topology of $G$. Since subgroups of precompact groups are precompact, we may and will assume that $G=\hull{A}$. By Theorem \ref{thm:main}, 

\begin{equation}\label{kal:je:iso}
\kal{A}:S_A\to (G,\tau) \mbox{ is an isomorphism and a homeomorphism. }
\end{equation}
By Fact \ref{fact:precompact:top}, characters of $(G,\tau)^*$ separate points of $G$. Therefore, for each $a\in A$ we may pick $\chi_a\in (G,\tau)^*$ such that $\chi_a(a)\neq 1$. The set $A$ is independent as it is topologically independent in $(G,\tau)$. Hence there is (unique) $\chi\in G'$ such that 
\begin{equation}\label{eq:chia:chi}
\mbox{
$\chi(a)=\chi_a(a)\neq 1$ for all $a\in A$.}
\end{equation}  Consider the subgroup $H$ of $G'$ generated by $(G,\tau)^*\cup\{\chi\}$. Then $A$ as a subset of $(G,T_H)$ is as required.  

Indeed, the topology $T_H$ is precompact by Fact \ref{fact:precompact:top} (it separates points as $(G,\tau)^*$ does). Further, since $A$ is infinite, the character $\chi$ is  not finitely $A$-supported by \eqref{eq:chia:chi}. Thus the set $A$ is not topologically independent in $(G,T_H)$  by Proposition \ref{prop:finitely:supported:characters}. 

Finally, it follows from  \eqref{eq:chia:chi} that each $\hull{a}$ has the same subgroup topology in the topology $T_H$ as in $\tau$. Therefore, $S_A$ as the Tychonoff direct sum has the same topology with respect to $T_H$ as well as with respect to $\tau$. Since $T_H$ is finer then $\tau$, it follows from \eqref{kal:je:iso}  that  the Kalton map $$\kal{A}:S_A\to (G,T_H)$$ is an open isomorphism.   
\end{example}

Call a group $G$ monothetic if it has a dense cyclic subgroup $C$. Every generator of $C$ is called a {\em topological generator} of $G$. 

\begin{lemma}\label{dense:means:not:top:indep}
 If  $g$  is a topological generator of a topological group $G$, then the singleton $\{g\}$ is a maximal (with respect to inclusion) topologically independent subset of $G$. 
\end{lemma}
\begin{proof}
Let $g,h$ be topologically independent elements of $G$. Since finite sets are absolutely Cauchy summable, Fact \ref{fact:basic} gives us that $\hull{\{g,h\}}$ is the Tychonoff direct sum $\hull{g}\bigoplus\hull{h}$. Hence $\hull{g}$ is not dense in $\hull{\{g,h\}}$ and consequently it is not dense in $G$ as well. Thus $g$ is not a topological generator of $G$.
\end{proof}

Our last and simple example provides a  warning showing that unlike in Theorem \ref{thm:main}, some other properties that hold for independent sets may not be transfered to the realm of topologically independent sets in precompact groups. 

\begin{example}
{\em  Obviously, if $a,b, c$ are distinct non-torsion elements of a group such that the set $\{a,c\}$ as well as the set $\{b,c\}$ is not independent, then also the set $\{a,b\}$ is not independent. On the other hand, there exist a compact topological group $G$ and non-torsion elements $a,b,c\in G$ such that the set $\{a,c\}$ as well as the set $\{b,c\}$ is not topologically independent while the set $\{a,b\}$ is topologically independent.}

Put $G=\S^2$, and $a=[g,1], b=[1,g]$, where $g\in\S$ is non-torsion. The set $\{a,b\}$ is topologically independent by Fact \ref{fact:basic}. It is a folklore fact, that $\S^2$ is monothetic. Let  $c\in\S^2$  be its topological generator.  Then $c$ is non-torsion and the set $\{a,c\}$ as well as the set $\{b,c\}$ is not topologically independent by Lemma \ref{dense:means:not:top:indep}.

\end{example}

\end{document}